\documentclass{amsart}
\usepackage{amssymb}
\usepackage{amsmath}
\usepackage{amsfonts}

\newtheorem{theorem}{Theorem}
\theoremstyle{plain}

\newtheorem{corollary}{Corollary}

\newtheorem{definition}{Definition}

\newtheorem{lemma}{Lemma}

\newtheorem{proposition}{Proposition}

\numberwithin{equation}{section}
\input{tcilatex}

\begin{document}
\title[Some new weighted compact embeddings results]{Some new weighted
compact embeddings results and existence of weak solutions for eigenvalue
Robin problem}
\author{Ismail AYDIN}
\address{Sinop University\\
Faculty of Arts and Sciences\\
Department of Mathematics\\
Sinop, Turkey}
\email{iaydin@sinop.edu.tr}
\author{Cihan UNAL}
\address{Assessment, Selection and Placement Center, Ankara, Turkey}
\email{cihanunal88@gmail.com}
\thanks{}
\dedicatory{}
\thanks{}

\begin{abstract}
By applying Mountain Pass Lemma, Ekeland's and Ricceri's variational
principle, Fountain Theorem, we prove the existence and multiplicity of
solutions for the following Robin problem%
\begin{equation*}
\left\{ 
\begin{array}{cc}
-\func{div}\left( a(x)\left\vert \nabla u\right\vert ^{p(x)-2}\nabla
u\right) =\lambda b(x)\left\vert u\right\vert ^{q(x)-2}u, & x\in \Omega \\ 
a(x)\left\vert \nabla u\right\vert ^{p(x)-2}\frac{\partial u}{\partial
\upsilon }+\beta (x)\left\vert u\right\vert ^{p(x)-2}u=0, & x\in \partial
\Omega ,%
\end{array}%
\right.
\end{equation*}%
under some appropriate conditions.in the space $W_{a,b}^{1,p(.)}\left(
\Omega \right) .$
\end{abstract}

\maketitle

\section{Introduction}

Let $a(x)$ and $b(x)$ be weight functions. The purpose of the present paper
is to study the following Robin problem involving weighted $p(x)$-Laplacian
of the form%
\begin{equation}
\left\{ 
\begin{array}{cc}
-\func{div}\left( a(x)\left\vert \nabla u\right\vert ^{p(x)-2}\nabla
u\right) =\lambda b(x)\left\vert u\right\vert ^{q(x)-2}u, & x\in \Omega \\ 
a(x)\left\vert \nabla u\right\vert ^{p(x)-2}\frac{\partial u}{\partial
\upsilon }+\beta (x)\left\vert u\right\vert ^{p(x)-2}u=0, & x\in \partial
\Omega ,%
\end{array}%
\right.  \label{P}
\end{equation}%
where $\Omega \subset 
\mathbb{R}
^{N}$ $(N\geq 2)$ is a bounded smooth domain, $\frac{\partial u}{\partial
\upsilon }$ is the outer unit normal derivative of $u$ with respect to $%
\partial \Omega $, $\lambda >0$ is a real number, $p,q$ are continuous
functions on $\overline{\Omega }$, i.e. $p,q\in C\left( \overline{\Omega }%
\right) $ with $\underset{x\in \overline{\Omega }}{\inf }p(x)>1$, and $\beta
\in L^{\infty }\left( \partial \Omega \right) $ such that $\beta ^{-}=%
\underset{x\in \partial \Omega }{\inf }\beta (x)>0$.

In recent years, variable exponent function spaces, especially Sobolev
spaces, have been very useful for exploring the existence of solutions of
partial differential equations involving $p(x)$-Laplacian. There are also
several applications in this area, such as electrorheological fluids,\ image
processing, elastic mechanics, fluid dynamics and calculus of variations,
see \cite{Ha}, \cite{mih}, \cite{Ru}, \cite{Zh}. Nonlinear and variational
problems involving $p(x)$-Laplacian operator have attracted great attention
in the recent last years, see \cite{ayd}, \cite{ay}, \cite{aydn}, \cite{den}%
, \cite{fan}, \cite{mi}, \cite{UnalAy2}.

In 2009, Deng \cite{deng} studied the following Robin problem%
\begin{equation}
\left\{ 
\begin{array}{cc}
-\func{div}\left( \left\vert \nabla u\right\vert ^{p(x)-2}\nabla u\right)
=\lambda f\left( x,u\right) , & x\in \Omega \\ 
\left\vert \nabla u\right\vert ^{p(x)-2}\frac{\partial u}{\partial \upsilon }%
+\beta (x)\left\vert u\right\vert ^{p(x)-2}u=0, & x\in \partial \Omega ,%
\end{array}%
\right.  \label{P2}
\end{equation}%
under appropriate conditions on $f$, and obtained that there exists $\lambda
^{\ast }>0$ such that the problem (\ref{P2}) has at least two positive
solutions if $\lambda \in \left( 0,\lambda ^{\ast }\right) $, has at least
one positive solution if $\lambda =\lambda ^{\ast }<+\infty $ and has no
solution if $\lambda >\lambda ^{\ast }$. Moreover, Kefi \cite{ke}
investigated the Robin problem 
\begin{equation}
\left\{ 
\begin{array}{cc}
-\func{div}\left( \left\vert \nabla u\right\vert ^{p(x)-2}\nabla u\right)
=\lambda V(x)\left\vert u\right\vert ^{q(x)-2}u, & x\in \Omega \\ 
\left\vert \nabla u\right\vert ^{p(x)-2}\frac{\partial u}{\partial \upsilon }%
+\beta (x)\left\vert u\right\vert ^{p(x)-2}u=0, & x\in \partial \Omega ,%
\end{array}%
\right.  \label{P3}
\end{equation}%
where $V\in L^{s(.)}\left( \partial \Omega \right) $ and $V>0$ with $%
1<q(x)<p(x)<s(x)<N$ for all $x\in \overline{\Omega }.$ Allaoui \cite{al}
obtained two different situations of the existence of solutions for the
problem (\ref{P3}) using variational methods. Chung \cite{Ch} have more
general results from Deng \cite{deng} and Kefi \cite{ke} under some suitable
assumptions for $q$ and $V$.

Inspired by the articles mentioned above, we will obtain some new compact
embedding theorems in double weighted variable exponent Sobolev spaces $%
W_{a,b}^{1,p(.)}\left( \Omega \right) $, and show the existence of several
different weak solutions of the problem (\ref{P}) in these spaces.
Therefore, we will investigate and generalize the results presented in
Allaoui \cite{al}, Chung \cite{Ch}, Deng \cite{deng} and Kefi \cite{ke}.

In section 2, we will mention about several properties of $%
W_{a,b}^{1,p(.)}\left( \Omega \right) $ and energy functional $J_{\lambda }$%
. Moreover, we will define a more general norm in $W_{a,b}^{1,p(.)}\left(
\Omega \right) $ compared to the norm given by Deng \cite{deng} using
compact embedding theorems. In section 3, we will investigate and prove the
main results of the paper.

\section{Notation and preliminaries}

Let $\Omega $ be a bounded open subset of $%
\mathbb{R}
^{N}$ with a smooth boundary $\partial \Omega $ and set 
\begin{equation*}
C_{+}\left( \overline{\Omega }\right) =\left\{ p\in C\left( \overline{\Omega 
}\right) :\inf_{x\in \overline{\Omega }}p(x)>1\right\} .
\end{equation*}%
For any $p\in C_{+}\left( \overline{\Omega }\right) $, we denote%
\begin{equation*}
1<p^{-}=\inf_{x\in \Omega }p(x)\leq p^{+}=\sup_{x\in \Omega }p(x)<\infty .
\end{equation*}%
Let $p\in C_{+}\left( \overline{\Omega }\right) $. The space $%
L^{p(.)}(\Omega )$ is defined by%
\begin{equation*}
L^{p(.)}(\Omega )=\left\{ u\left\vert u:\Omega \rightarrow 
\mathbb{R}
\text{ is measurable and }\int\limits_{\Omega }\left\vert u(x)\right\vert
^{p(x)}dx<\infty \right. \right\}
\end{equation*}%
with the (Luxemburg) norm%
\begin{equation*}
\left\Vert u\right\Vert _{p(.)}=\inf \left\{ \lambda >0:\varrho
_{p(.)}\left( \frac{u}{\lambda }\right) \leq 1\right\}
\end{equation*}%
where 
\begin{equation*}
\varrho _{p(.)}(u)=\int\limits_{\Omega }\left\vert u(x)\right\vert ^{p(x)}dx%
\text{,}
\end{equation*}%
see \cite{ko}. A measurable and locally integrable function $a:\Omega
\longrightarrow \left( 0,\infty \right) $ is called a weight function.
Define the weighted variable exponent Lebesgue space by%
\begin{equation*}
L_{a}^{p(.)}(\Omega )=\left\{ u\left\vert u:\Omega \longrightarrow 
\mathbb{R}
\text{ measurable and }\int\limits_{\Omega }\left\vert u(x)\right\vert
^{p(x)}a(x)dx<+\infty \right. \right\}
\end{equation*}%
with the Luxemburg norm 
\begin{equation*}
\left\Vert u\right\Vert _{p(.),a}=\inf \left\{ \tau >0:\varrho
_{p(.),a}\left( \frac{u}{\tau }\right) \leq 1\right\}
\end{equation*}%
where 
\begin{equation*}
\varrho _{p(.),a}(u)=\int\limits_{\Omega }\left\vert u(x)\right\vert
^{p(x)}a(x)dx\text{.}
\end{equation*}%
The space $L_{a}^{p(.)}(\Omega )$ is a Banach space with respect to $%
\left\Vert .\right\Vert _{p(.),a}.$ Moreover, $u\in L_{a}^{p(.)}(\Omega )$
if and only if $\left\Vert u\right\Vert _{p(.),a}=\left\Vert ua^{\frac{1}{%
p(.)}}\right\Vert _{p(.)}<\infty $. The dual space of $L_{a}^{p(.)}(\Omega )$
is $L_{a^{\ast }}^{q(.)}(\Omega )$ where $\frac{1}{p(.)}+\frac{1}{q(.)}=1$
and $a^{\ast }=a^{1-q\left( .\right) }=a^{-\frac{1}{p(.)-1}}.$ If $a\in
L^{\infty }\left( \Omega \right) $, then $L_{a}^{p(.)}=L^{p(.)},$ see \cite%
{ayn}.

\begin{proposition}
(see \cite{FanZ}) For all $u,v\in L_{a}^{p(.)}\left( \Omega \right) $, we
have

\begin{enumerate}
\item[\textit{(i)}] $\left\Vert u\right\Vert _{p(.),a}<1$ (resp.$=1,>1$) if
and only if $\varrho _{p(.),a}(u)<1$ (resp.$=1,>1$),

\item[\textit{(ii)}] $\left\Vert u\right\Vert _{p(.),a}^{p^{-}}\leq \varrho
_{p(.),a}(u)\leq \left\Vert u\right\Vert _{p(.),a}^{p^{+}}$ with $\left\Vert
u\right\Vert _{p(.),a}>1,$

\item[\textit{(iii)}] $\left\Vert u\right\Vert _{p(.),a}^{p^{+}}\leq \varrho
_{p(.),a}(u)\leq \left\Vert u\right\Vert _{p(.),a}^{p^{-}}$ with $\left\Vert
u\right\Vert _{p(.),a}<1$

\item[\textit{(iv)}] $\min \left\{ \left\Vert u\right\Vert
_{p(.),a}^{p^{-}},\left\Vert u\right\Vert _{p(.),a}^{p^{+}}\right\} \leq
\varrho _{p(.),a}(u)\leq \max \left\{ \left\Vert u\right\Vert
_{p(.),a}^{p^{-}},\left\Vert u\right\Vert _{p(.),a}^{p^{+}}\right\} $,

\item[\textit{(v)}] $\min \left\{ \varrho _{p(.),a}(u)^{\frac{1}{p^{-}}%
},\varrho _{p(.),a}(u)^{\frac{1}{p^{+}}}\right\} \leq \left\Vert
u\right\Vert _{p(.),a}\leq \max \left\{ \varrho _{p(.),a}(u)^{\frac{1}{p^{-}}%
},\varrho _{p(.),a}(u)^{\frac{1}{p^{+}}}\right\} $,

\item[\textit{(vi)}] $\varrho _{p(.),a}(u-v)\rightarrow 0$ if and only if $%
\left\Vert u-v\right\Vert _{p(.),a}\rightarrow 0$.
\end{enumerate}
\end{proposition}

\begin{definition}
Let $a^{-\frac{1}{p(.)-1}}\in L_{loc}^{1}\left( \Omega \right) $. The
weighted variable exponent Sobolev space $W_{a}^{k,p(.)}\left( \Omega
\right) $ is defined by%
\begin{equation*}
W_{a}^{k,p(.)}\left( \Omega \right) =\left\{ u\in L_{a}^{p(.)}(\Omega
):D^{\alpha }u\in L_{a}^{p(.)}(\Omega ),0\leq \left\vert \alpha \right\vert
\leq k\right\}
\end{equation*}%
equipped with the norm 
\begin{equation*}
\left\Vert u\right\Vert _{a}^{k,p(.)}=\sum\limits_{0\leq \left\vert \alpha
\right\vert \leq k}\left\Vert D^{\alpha }u\right\Vert _{p(.),a}
\end{equation*}%
where $\alpha \in 
\mathbb{N}
_{0}^{N}$ is a multi-index, $\left\vert \alpha \right\vert =\alpha
_{1}+\alpha _{2}+...+\alpha _{N}$ and $D^{\alpha }=\frac{\partial
^{\left\vert \alpha \right\vert }}{\partial _{x_{1}}^{\alpha
_{1}}...\partial _{x_{N}}^{\alpha _{N}}}$. In particular, the space $%
W_{a}^{1,p(.)}\left( \Omega \right) $ is defined by 
\begin{equation*}
W_{a}^{1,p(.)}\left( \Omega \right) =\left\{ u\in L_{a}^{p(.)}(\Omega
):\left\vert \nabla u\right\vert \in L_{a}^{p(.)}(\Omega )\right\}
\end{equation*}%
equipped with the norm 
\begin{equation*}
\left\Vert u\right\Vert _{a}^{1,p(.)}=\left\Vert u\right\Vert
_{p(.),a}+\left\Vert \nabla u\right\Vert _{p(.),a}.
\end{equation*}%
The dual space of $W_{a}^{1,p(.)}\left( \Omega \right) $ is $W_{a^{\ast
}}^{-1,r(.)}\left( \Omega \right) $ where $\frac{1}{p(.)}+\frac{1}{r(.)}=1$
and $a^{\ast }=a^{1-r\left( .\right) }=a^{-\frac{1}{p(.)-1}}$. Moreover, the
space $W_{a}^{1,p(.)}\left( \Omega \right) $ is a separable and reflexive
Banach space.
\end{definition}

Let $a^{-\frac{1}{p(.)-1}},b^{-\frac{1}{p(.)-1}}\in L_{loc}^{1}\left( \Omega
\right) $. The double weighted variable exponent Sobolev space $%
W_{a,b}^{1,p(.)}\left( \Omega \right) $ is defined by%
\begin{equation*}
W_{a,b}^{1,p(.)}\left( \Omega \right) =\left\{ u\in L_{b}^{p(.)}(\Omega
):\left\vert \nabla u\right\vert \in L_{a}^{p(.)}(\Omega )\right\}
\end{equation*}%
equipped with the norm%
\begin{equation*}
\left\Vert u\right\Vert _{a,b}^{1,p(.)}=\left\Vert \nabla u\right\Vert
_{p(.),a}+\left\Vert u\right\Vert _{p(.),b}.
\end{equation*}

We can define the space $L_{a}^{p(.)}(\partial \Omega )$ similarly by%
\begin{equation*}
L_{a}^{p(.)}(\partial \Omega )=\left\{ u\left\vert u:\partial \Omega
\longrightarrow 
\mathbb{R}
\text{ measurable and }\int\limits_{\partial \Omega }\left\vert
u(x)\right\vert ^{p(x)}a(x)d\sigma <+\infty \right. \right\}
\end{equation*}%
with the Luxemburg norm, where $d\sigma $ is the measure on the boundary of $%
\Omega $. It is easy to see that $\left( L_{a}^{p(.)}(\partial \Omega
),\left\Vert .\right\Vert _{p(.),a,\partial \Omega }\right) $ is a Banach
space. If $a\in L^{\infty }\left( \Omega \right) $, then $%
L_{a}^{p(.)}=L^{p(.)}.$

\begin{proposition}
\label{pro2}(see \cite{Liu})Let $\Omega $ be a (bounded or unbounded) domain
in $%
\mathbb{R}
^{N}$ with smooth boundary, $p,q\in L_{+}^{\infty }\left( \Omega \right) ,$ $%
p\in C^{0,1}\left( \overline{\Omega }\right) $, $p^{+}<N$, $1<q(x)<\infty $
and $-N<\alpha _{1}(x)<N\left( p(x)-1\right) .$ If $\underset{x\in \Omega }{%
\text{esssup}}\left( \frac{\alpha _{1}(x)}{q(x)}+\frac{N}{q(x)}-\frac{N}{p(x)%
}+1\right) <0$ and $\underset{x\in \Omega }{\text{essinf}}\left( \frac{N}{%
q(x)}-\frac{N}{p(x)}+1\right) >0$, then the embedding $W_{a,b}^{1,p(.)}%
\left( \Omega \right) \hookrightarrow L_{b}^{q(.)}(\Omega )$ is compact.
\end{proposition}

For $A\subset \overline{\Omega }$, denote $p^{-}(A)=\underset{x\in A}{\text{%
inf}}p(x)$ and $p^{+}(A)=\underset{x\in A}{\text{sup}}p(x)$. For any $x\in
\partial \Omega $ and $r\in C\left( \partial \Omega ,%
\mathbb{R}
\right) $ with $r^{-}=\underset{x\in \partial \Omega }{\text{inf}}r(x)>1,$
we define%
\begin{equation*}
p^{\partial }\left( x\right) =\left( p(x)\right) ^{\partial }=\left\{ 
\begin{array}{cc}
\frac{\left( N-1\right) p(x)}{N-p(x)}, & \text{if }p\left( x\right) <N, \\ 
\infty , & \text{if }p\left( x\right) \geq N,%
\end{array}%
\right.
\end{equation*}%
\begin{equation*}
p_{r(x)}^{\partial }\left( x\right) =\frac{r(x)-1}{r(x)}p^{\partial }\left(
x\right) .
\end{equation*}

\begin{theorem}
\label{teo2}(see \cite{de})Assume that the boundary of $\Omega $ possesses
the cone property and $p\in C\left( \overline{\Omega }\right) $ with $%
p^{-}>1 $. Suppose that $a\in L^{r(.)}(\partial \Omega )$, $r\in C\left(
\partial \Omega \right) $ with $r(x)>\frac{p^{\partial }\left( x\right) }{%
p^{\partial }\left( x\right) -1}$ for all $x\in \partial \Omega $. If $q\in
C\left( \partial \Omega \right) $ and $1\leq q(x)<p_{r(x)}^{\partial }\left(
x\right) $ for all $x\in \partial \Omega $, then there the embedding $%
W^{1,p(.)}\left( \Omega \right) \hookrightarrow L_{a}^{q(.)}(\partial \Omega
)$ is compact. In particular, the embedding $W^{1,p(.)}\left( \Omega \right)
\hookrightarrow L^{q(.)}(\partial \Omega )$ is compact where $1\leq
q(x)<p^{\partial }\left( x\right) $ for all $x\in \partial \Omega $.
\end{theorem}

It is easy to see that $p_{r(x)}^{\partial }\left( x\right) <p^{\partial
}\left( x\right) $ and $p\left( x\right) <p^{\partial }\left( x\right) $.
Thus we have the following result under the same assumptions in Theorem \ref%
{teo2}.

\begin{corollary}
(see \cite{de})

\begin{enumerate}
\item[\textit{(i)}] The embedding $W^{1,p(.)}\left( \Omega \right)
\hookrightarrow L^{p(.)}(\partial \Omega )$ is compact where $1\leq
p(x)<p^{\partial }\left( x\right) $ for all $x\in \partial \Omega $.

\item[\textit{(ii)}] The embedding $W^{1,p(.)}\left( \Omega \right)
\hookrightarrow L_{a}^{p(.)}(\partial \Omega )$ is compact where $1\leq
p(x)<p_{r(x)}^{\partial }\left( x\right) <p^{\partial }\left( x\right) $ for
all $x\in \partial \Omega $.
\end{enumerate}
\end{corollary}

\begin{theorem}
\label{teo3}(see \cite{aydn}) Let $a^{-\alpha \left( .\right) }\in
L^{1}\left( \Omega \right) $ with $\alpha \left( x\right) \in \left( \frac{N%
}{p(x)},\infty \right) \cap \left[ \frac{1}{p(x)-1},\infty \right) $. Then
the embedding $W_{a,b}^{1,p(.)}\left( \Omega \right) \hookrightarrow
W^{1,p_{\ast }(.)}\left( \Omega \right) $ is compact where $p_{\ast }(x)=%
\frac{\alpha \left( x\right) p(x)}{\alpha \left( x\right) +1}$.

\begin{corollary}
\label{cor2}If $p(x)<p_{\ast ,r(x)}^{\partial }\left( x\right) <p_{\ast
}^{\partial }\left( x\right) $ for all $x\in \partial \Omega $, then the
embedding $W_{a,b}^{1,p(.)}\left( \Omega \right) \hookrightarrow
L_{a}^{p(.)}(\partial \Omega )$ is compact.
\end{corollary}
\end{theorem}

If we use the same method in \cite[Theorem 2.1]{deng}, then we obtain the
following theorem which plays an important role for the existence of weak
solutions of the problem (\ref{P}).

\begin{theorem}
Let $\beta \in L^{\infty }\left( \partial \Omega \right) $ such that $\beta
^{-}=\underset{x\in \partial \Omega }{\inf }\beta (x)>0$. Then, the norm $%
\left\Vert u\right\Vert _{\beta (x)}$ is defined by%
\begin{equation*}
\left\Vert u\right\Vert _{\beta (x)}=\inf \left\{ \tau
>0:\int\limits_{\Omega }a(x)\left\vert \frac{\nabla u(x)}{\tau }\right\vert
^{p(x)}dx+\int\limits_{\partial \Omega }\beta (x)\left\vert \frac{u(x)}{\tau 
}\right\vert ^{p(x)}d\sigma \leq 1\right\}
\end{equation*}%
for any $u\in W_{a,b}^{1,p(.)}\left( \Omega \right) $. Moreover, $\left\Vert
.\right\Vert _{\beta (x)}$ and $\left\Vert .\right\Vert _{a,b}^{1,p(.)}$ are
equivalent on $W_{a,b}^{1,p(.)}\left( \Omega \right) $.
\end{theorem}

\begin{proposition}
(see \cite{deng}) Let $I_{\beta (x)}(u)=\int\limits_{\Omega }a(x)\left\vert
\nabla u(x)\right\vert ^{p(x)}dx+\int\limits_{\partial \Omega }\beta
(x)\left\vert u(x)\right\vert ^{p(x)}d\sigma $ with $\beta ^{-}>0$. For any $%
u,u_{k}\in W_{a,b}^{1,p(.)}\left( \Omega \right) $ $\left( k=1,2,...\right) $%
, we have

\begin{enumerate}
\item[\textit{(i)}] $\left\Vert u\right\Vert _{\beta (x)}^{p^{-}}\leq
I_{\beta (x)}(u)\leq \left\Vert u\right\Vert _{\beta (x)}^{p^{+}}$ with $%
\left\Vert u\right\Vert _{\beta (x)}\geq 1,$

\item[\textit{(ii)}] $\left\Vert u\right\Vert _{\beta (x)}^{p^{+}}\leq
I_{\beta (x)}(u)\leq \left\Vert u\right\Vert _{\beta (x)}^{p^{-}}$ with $%
\left\Vert u\right\Vert _{\beta (x)}\leq 1,$

\item[\textit{(iii)}] $\min \left\{ \left\Vert u\right\Vert _{\beta
(x)}^{p^{-}},\left\Vert u\right\Vert _{\beta (x)}^{p^{+}}\right\} \leq
I_{\beta (x)}(u)\leq \max \left\{ \left\Vert u\right\Vert _{\beta
(x)}^{p^{-}},\left\Vert u\right\Vert _{\beta (x)}^{p^{+}}\right\} ,$

\item[\textit{(iv)}] $\left\Vert u-u_{k}\right\Vert _{\beta (x)}\rightarrow
0 $ if and only if $I_{\beta (x)}(u-u_{k})\rightarrow 0$ as $k\rightarrow
\infty $,

\item[\textit{(v)}] $\left\Vert u_{k}\right\Vert _{\beta (x)}\rightarrow
\infty $ if and only if $I_{\beta (x)}(u_{k})\rightarrow \infty $ as $%
k\rightarrow \infty $.
\end{enumerate}
\end{proposition}

The proof of the following Proposition can be found in \cite[Proposition 2.2]%
{be}.

\begin{proposition}
\label{pro4}Let us define the functional $L_{\beta
(x)}:W_{a,b}^{1,p(.)}\left( \Omega \right) \rightarrow 
\mathbb{R}
$ by%
\begin{equation*}
L_{\beta (x)}\left( u\right) =\dint\limits_{\Omega }\frac{a(x)}{p\left(
x\right) }\left\vert \nabla u\right\vert ^{p\left( x\right)
}dx+\int\limits_{\partial \Omega }\frac{\beta (x)}{p(x)}\left\vert
u(x)\right\vert ^{p(x)}d\sigma
\end{equation*}%
for all $u\in W_{a,b}^{1,p(.)}\left( \Omega \right) $. Then we obtain $%
L_{\beta (x)}\in C^{1}\left( W_{a,b}^{1,p(.)}\left( \Omega \right) ,%
\mathbb{R}
\right) $ and 
\begin{equation*}
L_{\beta (x)}^{\prime }\left( u\right) (v)=<L_{\beta (x)}^{\prime }\left(
u\right) ,v>=\dint\limits_{\Omega }a(x)\left\vert \nabla u\right\vert
^{p\left( x\right) -2}\nabla u\nabla vdx+\int\limits_{\partial \Omega }\beta
(x)\left\vert u(x)\right\vert ^{p(x)-2}uvd\sigma
\end{equation*}%
for any $u,v\in W_{a,b}^{1,p(.)}\left( \Omega \right) $. In addition, we
have the following properties

\begin{enumerate}
\item[\textit{(i)}] $L_{\beta (x)}^{\prime }:W_{a,b}^{1,p(.)}\left( \Omega
\right) \longrightarrow W_{a^{\ast },b^{\ast }}^{-1,p^{\prime }(.)}\left(
\Omega \right) $ is continuous, bounded and strictly monotone operator,

\item[\textit{(ii)}] $L_{\beta (x)}^{\prime }:W_{a,b}^{1,p(.)}\left( \Omega
\right) \longrightarrow W_{a^{\ast },b^{\ast }}^{-1,p^{\prime }(.)}\left(
\Omega \right) $ is a mapping of type $\left( S_{+}\right) ,$ i.e., if $%
u_{n}\rightharpoonup u$ in $W_{a,b}^{1,p(.)}\left( \Omega \right) $ and $%
\underset{n\longrightarrow \infty }{\limsup }L_{\beta (x)}^{\prime }\left(
u_{n}\right) (u_{n}-u)\leq 0$, then $u_{n}\longrightarrow u$ in $%
W_{a,b}^{1,p(.)}\left( \Omega \right) $

\item[\textit{(iii)}] $L_{\beta (x)}^{\prime }:W_{a,b}^{1,p(.)}\left( \Omega
\right) \longrightarrow W_{a^{\ast },b^{\ast }}^{-1,p^{\prime }(.)}\left(
\Omega \right) $ is a homeomorphism.
\end{enumerate}
\end{proposition}

\begin{definition}
We call that $u\in W_{a,b}^{1,p(.)}\left( \Omega \right) $ is a weak
solution of the problem (\ref{P}) if%
\begin{equation*}
\dint\limits_{\Omega }a(x)\left\vert \nabla u\right\vert ^{p\left( x\right)
-2}\nabla u\nabla vdx+\int\limits_{\partial \Omega }\beta (x)\left\vert
u(x)\right\vert ^{p(x)-2}uvd\sigma -\lambda \tint\limits_{\Omega
}b(x)\left\vert u\right\vert ^{q(x)-2}uvdx=0
\end{equation*}%
for all $v\in W_{a,b}^{1,p(.)}\left( \Omega \right) .$ We point out that if $%
\lambda \in 
\mathbb{R}
$ is an eigenvalue of the problem (\ref{P}), then the corresponding $u\in
W_{a,b}^{1,p(.)}\left( \Omega \right) -\left\{ 0\right\} $ is a weak
solution of (\ref{P}).
\end{definition}

To find out a weak solution to (\ref{P}), let us introduce the functional $%
J_{\lambda }:W_{a,b}^{1,p(.)}\left( \Omega \right) \rightarrow 
\mathbb{R}
$ defined by%
\begin{equation*}
J_{\lambda }\left( u\right) =\dint\limits_{\Omega }\frac{a(x)}{p\left(
x\right) }\left\vert \nabla u\right\vert ^{p\left( x\right)
}dx+\int\limits_{\partial \Omega }\frac{\beta (x)}{p(x)}\left\vert
u(x)\right\vert ^{p(x)}d\sigma -\lambda \dint\limits_{\Omega }\frac{b\left(
x\right) }{q\left( x\right) }\left\vert u\right\vert ^{q\left( x\right) }dx
\end{equation*}%
for any $\lambda >0.$ It is easy to see that $J_{\lambda }$ is well defined
from $W_{a,b}^{1,p(.)}\left( \Omega \right) $ to $%
\mathbb{R}
$ and $J_{\lambda }\in C^{1}\left( W_{a,b}^{1,p(.)}\left( \Omega \right) ,%
\mathbb{R}
\right) $ with the derivative given by%
\begin{equation*}
J_{\lambda }^{\prime }\left( u\right) (v)=\dint\limits_{\Omega
}a(x)\left\vert \nabla u\right\vert ^{p\left( x\right) -2}\nabla u\nabla
vdx+\int\limits_{\partial \Omega }\beta (x)\left\vert u(x)\right\vert
^{p(x)-2}uvd\sigma -\lambda \dint\limits_{\Omega }b(x)\left\vert
u\right\vert ^{q(x)-2}uvdx
\end{equation*}%
for all $u,v\in W_{a,b}^{1,p(.)}\left( \Omega \right) .$ The operator $%
J_{\lambda }^{\prime }\left( u\right) $ belongs to the dual space $%
W_{a^{\ast },b^{\ast }}^{-1,p^{\prime }(.)}\left( \Omega \right) $. It is
also note that $u\in W_{a,b}^{1,p(.)}\left( \Omega \right) $ is a weak
solution of the problem (\ref{P}) if and only if $u$ is a critical point of $%
J_{\lambda }$.

\section{Main results}

Throughout this paper we assume that all assumptions in Proposition \ref%
{pro2} and Corollary \ref{cor2} are valid.

\begin{lemma}
\label{PScon}Let $p^{+}<\eta <q^{-}$. Then the operator $J_{\lambda }$
satisfies the (PS) condition.
\end{lemma}

\begin{proof}
Let $\left( u_{n}\right) _{n\in 
\mathbb{N}
}\subset W_{a,b}^{1,p(.)}\left( \Omega \right) $ be such that 
\begin{equation}
\sup \left\vert J_{\lambda }\left( u_{n}\right) \right\vert \leq M,\text{ \
\ }J_{\lambda }^{\prime }\left( u_{n}\right) \rightarrow 0\text{ in }%
W_{a^{\ast },b^{\ast }}^{-1,p^{\prime }(.)}\left( \Omega \right) \text{ as }%
n\rightarrow \infty ,  \label{P4}
\end{equation}%
i.e., $\left( u_{n}\right) _{n\in 
\mathbb{N}
}$ is a (PS) sequence. First, we will prove that $\left( u_{n}\right) _{n\in 
\mathbb{N}
}$ is bounded in $W_{a,b}^{1,p(.)}\left( \Omega \right) .$ For this, we
assume that $\left( u_{n}\right) _{n\in 
\mathbb{N}
}$ is not bounded. Thus, we can suppose that $\left\Vert u_{n}\right\Vert
_{\beta (x)}>1$ for all $n\in 
\mathbb{N}
.$ Therefore, we have%
\begin{eqnarray*}
&&M+\left\Vert u_{n}\right\Vert _{\beta (x)} \\
&\geq &J_{\lambda }\left( u_{n}\right) -\frac{1}{\eta }\left\langle
J_{\lambda }^{\prime }\left( u_{n}\right) ,u_{n}\right\rangle \\
&=&\dint\limits_{\Omega }\frac{a(x)}{p\left( x\right) }\left\vert \nabla
u_{n}\right\vert ^{p\left( x\right) }dx+\int\limits_{\partial \Omega }\frac{%
\beta (x)}{p(x)}\left\vert u_{n}(x)\right\vert ^{p(x)}d\sigma -\lambda
\dint\limits_{\Omega }\frac{b\left( x\right) }{q\left( x\right) }\left\vert
u_{n}\right\vert ^{q\left( x\right) }dx \\
&&-\frac{1}{\eta }\dint\limits_{\Omega }a(x)\left\vert \nabla
u_{n}\right\vert ^{p\left( x\right) }dx-\frac{1}{\eta }\int\limits_{\partial
\Omega }\beta (x)\left\vert u_{n}(x)\right\vert ^{p(x)}d\sigma +\frac{%
\lambda }{\eta }\dint\limits_{\Omega }b(x)\left\vert u_{n}\right\vert
^{q(x)}dx \\
&\geq &\left( \frac{1}{p^{+}}-\frac{1}{\eta }\right) I_{\beta
(x)}(u_{n})+\lambda \left( \frac{1}{\eta }-\frac{1}{q^{-}}\right)
\dint\limits_{\Omega }b(x)\left\vert u_{n}\right\vert ^{q(x)}dx \\
&\geq &\left( \frac{1}{p^{+}}-\frac{1}{\eta }\right) \left\Vert
u_{n}\right\Vert _{\beta (x)}^{p^{-}}.
\end{eqnarray*}%
Thus, we get $\eta \leq p^{+}$, which is a contradiction. This means that $%
\left( u_{n}\right) $ is bounded in $W_{a,b}^{1,p(.)}\left( \Omega \right) .$
Since $W_{a,b}^{1,p(.)}\left( \Omega \right) $ is a reflexive Banach space,
there exists $u\in W_{a,b}^{1,p(.)}\left( \Omega \right) $ such that $%
u_{n}\rightharpoonup u$ in $W_{a,b}^{1,p(.)}\left( \Omega \right) .$ By
Proposition \ref{pro2}, the embedding $W_{a,b}^{1,p(.)}\left( \Omega \right)
\hookrightarrow L_{b}^{q\left( .\right) }\left( \Omega \right) $ is compact.
Thus, we have $u_{n}\rightarrow u$ in $L_{b}^{q\left( .\right) }\left(
\Omega \right) $ and $\underset{n}{\sup }\left\Vert u_{n}\right\Vert _{\beta
(x)}<\infty $. Moreover, by the H\"{o}lder inequality, we obtain 
\begin{eqnarray*}
&&\left\vert \dint\limits_{\Omega }b(x)\left\vert u_{n}\right\vert
^{q(x)-2}u_{n}\left( u_{n}-u\right) dx\right\vert \\
&\leq &\dint\limits_{\Omega }b(x)^{\frac{1}{q(x)}+\frac{1}{r(x)}}\left\vert
u_{n}\right\vert ^{q(x)-1}\left\vert u_{n}-u\right\vert dx \\
&\leq &C\left\Vert \left\vert u_{n}\right\vert ^{q(x)-1}b(x)^{\frac{1}{r(x)}%
}\right\Vert _{r(.)}.\left\Vert \left\vert u_{n}-u\right\vert b(x)^{\frac{1}{%
q(x)}}\right\Vert _{q(.)} \\
&=&C\left\Vert u_{n}\right\Vert _{q(.),b}.\left\Vert u_{n}-u\right\Vert
_{q(.),b} \\
&\leq &C^{\ast }\left\Vert u_{n}\right\Vert _{\beta (x)}.\left\Vert
u_{n}-u\right\Vert _{q(.),b}\longrightarrow 0
\end{eqnarray*}%
where $\frac{1}{q\left( .\right) }+\frac{1}{r\left( .\right) }=1.$ This
follows that 
\begin{equation}
\lim_{n\longrightarrow \infty }\dint\limits_{\Omega }b(x)\left\vert
u_{n}\right\vert ^{q(x)-2}u_{n}\left( u_{n}-u\right) dx=0.  \label{P5}
\end{equation}%
By (\ref{P4}) and the boundedness of $\left\{ u_{n}-u\right\} $ in $%
W_{a,b}^{1,p(.)}\left( \Omega \right) $, we have 
\begin{equation}
J_{\lambda }^{\prime }\left( u_{n}\right) \left( u_{n}-u\right) \rightarrow
0.  \label{P6}
\end{equation}%
If we consider (\ref{P5}) and (\ref{P6}), then we get 
\begin{equation*}
\lim_{n\longrightarrow \infty }\left( \dint\limits_{\Omega }a(x)\left\vert
\nabla u_{n}\right\vert ^{p\left( x\right) -2}\nabla u_{n}\left( \nabla
u_{n}-\nabla u\right) dx+\int\limits_{\partial \Omega }\beta (x)\left\vert
u_{n}(x)\right\vert ^{p(x)-2}u_{n}\left( u_{n}-u\right) d\sigma \right) =0
\end{equation*}%
or equivalently 
\begin{equation}
\lim_{n\longrightarrow \infty }L_{\beta (x)}^{\prime }\left( u_{n}\right)
(u_{n}-u)=0.  \label{P7}
\end{equation}%
This follows by (\ref{P7}) and Proposition \ref{pro4} that the sequence $%
\left\{ u_{n}\right\} $ converges strongly to $u$ in $W_{a,b}^{1,p(.)}\left(
\Omega \right) $. This completes the proof.
\end{proof}

\begin{theorem}
\label{teo5}Let $p^{+}<\eta <q^{-}$. Then, the problem (\ref{P}) has a
nontrivial weak solution.
\end{theorem}

\begin{proof}
By Lemma \ref{PScon}, the functional $J_{\lambda }$ satisfies (PS) condition
on $W_{a,b}^{1,p(.)}\left( \Omega \right) $. Since $p^{+}<\eta <q^{-}$ and $%
W_{a,b}^{1,p(.)}\left( \Omega \right) \hookrightarrow L_{b}^{q\left(
.\right) }\left( \Omega \right) ,$ we have%
\begin{eqnarray*}
J_{\lambda }\left( u\right) &\geq &\frac{1}{p^{+}}\left(
\dint\limits_{\Omega }a(x)\left\vert \nabla u\right\vert ^{p\left( x\right)
}dx+\int\limits_{\partial \Omega }\beta (x)\left\vert u(x)\right\vert
^{p(x)}d\sigma \right) -\frac{\lambda }{q^{-}}\dint\limits_{\Omega
}b(x)\left\vert u\right\vert ^{q\left( x\right) }dx \\
&\geq &\frac{1}{p^{+}}I_{\beta (x)}(u)-\frac{\lambda }{q^{-}}\max \left\{
\left\Vert u\right\Vert _{q(.),b}^{q^{-}},\left\Vert u\right\Vert
_{q(.),b}^{q^{+}}\right\} \\
&\geq &\frac{1}{p^{+}}I_{\beta (x)}(u)-\frac{\lambda }{q^{-}}\max \left\{
\left\Vert u\right\Vert _{\beta (x)}^{q^{-}},\left\Vert u\right\Vert _{\beta
(x)}^{q^{+}}\right\} \\
&\geq &\frac{1}{p^{+}}\left\Vert u\right\Vert _{\beta (x)}^{p^{+}}-\frac{%
\lambda }{q^{-}}\left\Vert u\right\Vert _{\beta (x)}^{q^{-}} \\
&\geq &\left( \frac{1}{p^{+}}-\frac{\lambda }{q^{-}}\right) \left\Vert
u\right\Vert _{\beta (x)}^{p^{+}}
\end{eqnarray*}%
for $\left\Vert u\right\Vert _{\beta (x)}\leq 1$ and $p^{+}\lambda <q^{-}.$
Thus, when $\left\Vert u\right\Vert _{\beta (x)}=\rho $ sufficiently small,
we have $J_{\lambda }\left( u\right) >0.$ Moreover, since $p^{+}<\eta <q^{-}$%
, we get%
\begin{eqnarray*}
J_{\lambda }\left( tv\right) &\leq &t^{p^{+}}I_{\beta (x)}(v)-\lambda
t^{q^{-}}\dint\limits_{\Omega }b(x)\left\vert v\right\vert ^{q\left(
x\right) }dx \\
&\leq &t^{p^{+}}I_{\beta (x)}(v)-\lambda t^{\eta ^{-}}\dint\limits_{\Omega
}b(x)\left\vert v\right\vert ^{q\left( x\right) }dx\longrightarrow -\infty
\end{eqnarray*}%
as $t\longrightarrow \infty $ for $v\in W_{a,b}^{1,p(.)}\left( \Omega
\right) -\left\{ 0\right\} .$ It is note that $J_{\lambda }\left( 0\right)
=0.$ This follows that $J_{\lambda }$ satisfies the geometric conditions of
the Mountain Pass Theorem (see \cite{Wil}), and the operator $J_{\lambda }$
admits at least one nontrivial critical point.
\end{proof}

To apply Ekeland's variational principle to the problem (\ref{P}), we need
two following lemmas.

\begin{lemma}
\label{lem2}Let $1<q^{-}<p^{-}<q^{+}<p^{+}$. Then, there exists $\lambda
^{\ast }>0$ such that for any $\lambda \in \left( 0,\lambda ^{\ast }\right) $
there exist $\rho ,\gamma >0$ such that $J_{\lambda }(u)\geq \gamma >0$ for
any $u\in W_{a,b}^{1,p(.)}\left( \Omega \right) $ with $\left\Vert
u\right\Vert _{\beta (x)}=\rho .$
\end{lemma}

\begin{proof}
By $W_{a,b}^{1,p(.)}\left( \Omega \right) \hookrightarrow L_{b}^{q\left(
.\right) }\left( \Omega \right) $, there exists a constant $C_{2}>0$ such
that%
\begin{equation}
\left\Vert u\right\Vert _{q(.),b}\leq C_{2}\left\Vert u\right\Vert _{\beta
(x)}  \label{P8}
\end{equation}%
for any $u\in W_{a,b}^{1,p(.)}\left( \Omega \right) $. We take $\rho \in
\left( 0,1\right) $ such that $\rho <\frac{1}{C_{2}}$. Then (\ref{P8})
implies $\left\Vert u\right\Vert _{q(.),b}<1$ and 
\begin{equation}
\dint\limits_{\Omega }\left\vert u\right\vert ^{q\left( x\right) }b\left(
x\right) dx\leq \left\Vert u\right\Vert _{q(.),b}^{q^{-}}  \label{P9}
\end{equation}%
for all $u\in W_{a,b}^{1,p(.)}\left( \Omega \right) $ with $\left\Vert
u\right\Vert _{\beta (x)}=\rho .$ Combining (\ref{P8}) and (\ref{P9}), we
obtain%
\begin{equation}
\dint\limits_{\Omega }\left\vert u\right\vert ^{q\left( x\right) }b\left(
x\right) dx\leq C_{2}^{q^{-}}\left\Vert u\right\Vert _{\beta (x)}^{q^{-}}
\label{P10}
\end{equation}%
for all $u\in W_{a,b}^{1,p(.)}\left( \Omega \right) $ with $\left\Vert
u\right\Vert _{\beta (x)}=\rho .$ By (\ref{P10}), we deduce that 
\begin{eqnarray*}
J_{\lambda }\left( u\right) &\geq &\frac{1}{p^{+}}I_{\beta (x)}(u)-\frac{%
\lambda }{q^{-}}\dint\limits_{\Omega }b(x)\left\vert u\right\vert ^{q\left(
x\right) }dx \\
&\geq &\frac{1}{p^{+}}\left\Vert u\right\Vert _{\beta (x)}^{p^{+}}-\frac{%
\lambda }{q^{-}}C_{2}^{q^{-}}\left\Vert u\right\Vert _{\beta (x)}^{q^{-}} \\
&\geq &\frac{1}{p^{+}}\left\Vert u\right\Vert _{\beta (x)}^{p^{+}}-\frac{%
\lambda }{q^{-}}C_{2}^{q^{-}}\left\Vert u\right\Vert _{\beta (x)}^{q^{-}} \\
&=&\rho ^{q^{-}}\left( \frac{1}{p^{+}}\rho ^{p^{+}-q^{-}}-\frac{\lambda }{%
q^{-}}C_{2}^{q^{-}}\right) .
\end{eqnarray*}%
Set%
\begin{equation}
\lambda ^{\ast }=\frac{\rho ^{p^{+}-q^{-}}}{2p^{+}}.\frac{q^{-}}{%
C_{2}^{q^{-}}}.  \label{P11}
\end{equation}%
Then there exists $\gamma =\frac{\rho ^{p^{+}}}{2p^{+}}>0$ such that $%
J_{\lambda }(u)\geq \gamma >0$ for any $\lambda \in \left( 0,\lambda ^{\ast
}\right) $ and $u\in W_{a,b}^{1,p(.)}\left( \Omega \right) $ with $%
\left\Vert u\right\Vert _{\beta (x)}=\rho .$ This completes the proof.
\end{proof}

\begin{lemma}
\label{lem3}There is a function $\phi $ in $W_{a,b}^{1,p(.)}\left( \Omega
\right) $ satisfying $\phi \geq 0$, $\phi \neq 0$ and $J_{\lambda }\left(
t\phi \right) <0$ when $t$ is small enough for $1<q^{-}<p^{-}<q^{+}<p^{+}$.
\end{lemma}

\begin{proof}
Using the assumption $q^{-}<p^{-}$ we have $q^{-}+\varepsilon _{0}<p^{-}$
for a $\varepsilon _{0}>0$. Since $q\in C\left( \overline{\Omega }\right) $,
then we can find an open subset $\Omega _{0}\subset \Omega $ such that $%
\left\vert q(x)-q^{-}\right\vert <\varepsilon _{0}$ for any $x\in \Omega
_{0} $. Thus, we infer that%
\begin{equation}
q(x)\leq q^{-}+\varepsilon _{0}<p^{-}  \label{P12}
\end{equation}%
for all $x\in \Omega _{0}$. Let $\phi \in C_{0}^{\infty }\left( \Omega
\right) $ be such that $\overline{\Omega _{0}}\subset \sup $p$\phi $, $\phi
(x)=1$ for all $x\in \overline{\Omega _{0}}$ and $0\leq \phi \leq 1$ in $%
\Omega $. Without loss of generality, we may assume that $\left\Vert \phi
\right\Vert _{\beta (x)}=1$, that is 
\begin{equation}
I_{\beta (x)}(u)=\int\limits_{\Omega }a(x)\left\vert \nabla u(x)\right\vert
^{p(x)}dx+\int\limits_{\partial \Omega }\beta (x)\left\vert u(x)\right\vert
^{p(x)}d\sigma =1.  \label{P13}
\end{equation}%
Then using (\ref{P12}) and (\ref{P13}) for any $t\in \left( 0,1\right) ,$ we
have%
\begin{eqnarray*}
J_{\lambda }\left( t\phi \right) &\leq &\frac{t^{p^{-}}}{p^{-}}I_{\beta
(x)}(u)-\frac{\lambda }{q^{+}}\dint\limits_{\Omega }t^{q(x)}\left\vert \phi
\right\vert ^{q\left( x\right) }b\left( x\right) dx \\
&\leq &\frac{t^{p^{-}}}{p^{-}}-\frac{\lambda }{q^{+}}\dint\limits_{\Omega
_{0}}t^{q(x)}\left\vert \phi \right\vert ^{q\left( x\right) }b\left(
x\right) dx \\
&\leq &\frac{t^{p^{-}}}{p^{-}}-\frac{\lambda t^{q^{-}+\varepsilon _{0}}}{%
q^{+}}\dint\limits_{\Omega _{0}}\left\vert \phi \right\vert ^{q\left(
x\right) }b\left( x\right) dx.
\end{eqnarray*}%
Then, we obtain $J_{\lambda }\left( t\phi \right) <0$ for any $t<\delta ^{%
\frac{1}{p^{-}-q^{-}-\varepsilon _{0}}}$ with $0<\delta <\min \left\{ 1,%
\frac{\lambda p^{-}}{q^{+}}\dint\limits_{\Omega _{0}}\left\vert \phi
\right\vert ^{q\left( x\right) }b\left( x\right) dx\right\} .$ That is the
desired result.
\end{proof}

\begin{theorem}
\label{teo6}Let $1<q^{-}<p^{-}<q^{+}<p^{+}$. Then, there exists $\lambda
^{\ast }>0$ such that any $\lambda \in \left( 0,\lambda ^{\ast }\right) $ is
an eigenvalue for the problem (\ref{P}).
\end{theorem}

\begin{proof}
Let $\lambda \in \left( 0,\lambda ^{\ast }\right) $, where $\lambda ^{\ast }$
is defined as (\ref{P11}). By Lemma \ref{lem2} it follows that on the
boundary of the ball centered at the origin and of radius $\rho $ in $%
W_{a,b}^{1,p(.)}\left( \Omega \right) $, denoted by $B_{\rho }\left(
0\right) ,$ we have%
\begin{equation*}
\inf_{\partial B_{\rho }\left( 0\right) }J_{\lambda }>0.
\end{equation*}

On the other hand, from Lemma \ref{lem3}, there exists $\phi \in
W_{a,b}^{1,p(.)}\left( \Omega \right) $ such that $J_{\lambda }\left( t\phi
\right) <0$ for $t>0$ small enough. Moreover, we get%
\begin{equation*}
J_{\lambda }(u)\geq \frac{1}{p^{+}}\left\Vert u\right\Vert _{\beta
(x)}^{p^{+}}-\frac{\lambda }{q^{-}}C_{2}^{q^{-}}\left\Vert u\right\Vert
_{\beta (x)}^{q^{-}}
\end{equation*}%
for any $u\in W_{a,b}^{1,p(.)}\left( \Omega \right) .$ This follows that%
\begin{equation*}
-\infty <\underline{c}=\inf_{\overline{B_{\rho }\left( 0\right) }}J_{\lambda
}<0.
\end{equation*}%
Let choose $\varepsilon >0$. Therefore, we have%
\begin{equation*}
0<\varepsilon <\inf_{\partial B_{\rho }\left( 0\right) }J_{\lambda
}-\inf_{B_{\rho }\left( 0\right) }J_{\lambda }.
\end{equation*}%
Applying Ekeland's variational principle to the functional $J_{\lambda }:%
\overline{B_{\rho }\left( 0\right) }\longrightarrow 
\mathbb{R}
,$ we find $u_{\varepsilon }\in \overline{B_{\rho }\left( 0\right) }$ such
that%
\begin{eqnarray*}
J_{\lambda }\left( u_{\varepsilon }\right) &<&\inf_{\overline{B_{\rho
}\left( 0\right) }}J_{\lambda }+\varepsilon , \\
J_{\lambda }\left( u_{\varepsilon }\right) &<&J_{\lambda }\left( u\right)
+\varepsilon \left\Vert u-u_{\varepsilon }\right\Vert _{\beta (x)},\text{ }%
u\neq u_{\varepsilon }.
\end{eqnarray*}%
Since 
\begin{equation*}
J_{\lambda }\left( u_{\varepsilon }\right) \leq \inf_{\overline{B_{\rho
}\left( 0\right) }}J_{\lambda }+\varepsilon \leq \inf_{B_{\rho }\left(
0\right) }J_{\lambda }+\varepsilon <\inf_{\partial B_{\rho }\left( 0\right)
}J_{\lambda },
\end{equation*}%
we can infer that $u_{\varepsilon }\in B_{\rho }\left( 0\right) .$ Set a
function $\psi :\overline{B_{\rho }\left( 0\right) }\longrightarrow 
\mathbb{R}
$ such that $\psi \left( u\right) =J_{\lambda }(u)+\varepsilon \left\Vert
u-u_{\varepsilon }\right\Vert _{\beta (x)}.$ It is easy to see that $%
u_{\varepsilon }$ is a minimum point of $\psi .$ This follows 
\begin{equation*}
\frac{\psi \left( u_{\varepsilon }+tv\right) -\psi \left( u_{\varepsilon
}\right) }{t}\geq 0
\end{equation*}%
for small $t>0$ and any $v\in B_{\rho }\left( 0\right) $. From the above
relation we obtain%
\begin{equation*}
\frac{J_{\lambda }\left( u_{\varepsilon }+tv\right) -J_{\lambda }\left(
u_{\varepsilon }\right) }{t}+\varepsilon \left\Vert v\right\Vert _{\beta
(x)}\geq 0.
\end{equation*}%
By letting $t\longrightarrow 0,$ we get%
\begin{equation}
\left\langle J_{\lambda }^{\prime }\left( u_{\varepsilon }\right)
,v\right\rangle +\varepsilon \left\Vert v\right\Vert _{\beta (x)}\geq 0.
\label{P14}
\end{equation}%
Replacing $v$ by $-v$ in (\ref{P14}), we get $\left\vert \left\langle
J_{\lambda }^{\prime }\left( u_{\varepsilon }\right) ,v\right\rangle
\right\vert \leq \varepsilon \left\Vert v\right\Vert _{\beta (x)}$. It is
known that $J_{\lambda }^{\prime }\left( u_{\varepsilon }\right) \in
W_{a^{\ast },b^{\ast }}^{-1,p^{\prime }(.)}\left( \Omega \right) $. If we
use the definition of operator norm for $W_{a^{\ast },b^{\ast
}}^{-1,p^{\prime }(.)}\left( \Omega \right) $, we have $\left\Vert
J_{\lambda }^{\prime }\left( u_{\varepsilon }\right) \right\Vert
_{W_{a^{\ast },b^{\ast }}^{-1,p^{\prime }(.)}\left( \Omega \right) }\leq
\varepsilon .$ Therefore, we conclude that there is a sequence $\left\{
w_{n}\right\} \subset B_{\rho }\left( 0\right) $ such that%
\begin{equation}
J_{\lambda }\left( w_{n}\right) \longrightarrow \underline{c}\text{ and }%
J_{\lambda }^{\prime }\left( w_{n}\right) \longrightarrow 0.  \label{P15}
\end{equation}%
Since $J_{\lambda }$ satisfies PS-condition on $W_{a,b}^{1,p(.)}\left(
\Omega \right) ,$ we get that $\left\{ w_{n}\right\} $ converges strongly to 
$w$ in $W_{a,b}^{1,p(.)}\left( \Omega \right) $. This follows by (\ref{P15})
that 
\begin{equation*}
J_{\lambda }\left( w\right) =\underline{c}<0\text{ and }J_{\lambda }^{\prime
}\left( w\right) =0.
\end{equation*}%
We obtain that $w$ is a non-trivial weak solution for the problem (\ref{P})
and thus any $\lambda \in \left( 0,\lambda ^{\ast }\right) $ is an
eigenvalue of the problem (\ref{P}). That is the desired result.
\end{proof}

\begin{corollary}
There exists $\lambda ^{\ast }>0$ such that for any $\lambda \in \left(
0,\lambda ^{\ast }\right) $ the problem (\ref{P}) has at least two
non-negative non-trivial weak solutions.

\begin{proof}
By Theorem \ref{teo5}, we deduce a function $u\in W_{a,b}^{1,p(.)}\left(
\Omega \right) $ as a non-trivial critical point of the functional $%
J_{\lambda }$ with $J_{\lambda }\left( u\right) =\overline{c}>0$ and a
non-trivial weak solution of the problem (\ref{P}). Moreover, the problem (%
\ref{P}) has a non-trivial weak solution $w$ in $W_{a,b}^{1,p(.)}\left(
\Omega \right) $ with $J_{\lambda }\left( w\right) =\underline{c}<0$ by
Theorem \ref{teo6}. Finally, we point out the fact that $u\neq w$ since $%
J_{\lambda }\left( u\right) =\overline{c}>0>\underline{c}=J_{\lambda }\left(
w\right) .$ Therefore, the problem (\ref{P}) has at least two non-negative
non-trivial weak solutions due to $J_{\lambda }\left( u\right) =J_{\lambda
}\left( \left\vert u\right\vert \right) .$ This completes the proof.
\end{proof}
\end{corollary}

For the multiplicity solutions of the problem (\ref{P}), we will use some
well-known results.

\begin{lemma}
(see \cite{FanZ}) Let $X$ be a reflexive and separable Banach space. Then
there exist $\left\{ e_{j}\right\} \subset X$ and $\left\{ e_{j}^{\ast
}\right\} \subset X^{\ast }$ such that%
\begin{equation*}
X=\overline{span\left\{ e_{j}\left\vert j=1,2,...\right. \right\} },\text{ \
\ \ }X^{\ast }=\overline{span\left\{ e_{j}^{\ast }\left\vert
j=1,2,...\right. \right\} }
\end{equation*}%
and%
\begin{equation*}
\left\langle e_{i}^{\ast },e_{j}\right\rangle =\left\{ 
\begin{array}{c}
1,\text{ \ \ \ }i=j \\ 
0,\text{ \ \ \ }i\neq j%
\end{array}%
\right.
\end{equation*}%
where $\left\langle \cdot ,\cdot \right\rangle $ denotes the duality product
between $X$ and $X^{\ast }.$
\end{lemma}

For convenience, we write $X_{j}=span\left\{ e_{j}\right\} ,$ $%
Y_{k}=\dbigoplus\limits_{j=1}^{k}X_{j},$ $Z_{k}=\dbigoplus\limits_{j=k}^{%
\infty }X_{j}.$

\begin{lemma}
If we define the set 
\begin{equation*}
\alpha _{k}=\sup \left\{ \left\Vert u\right\Vert _{q\left( .\right)
,b}:\left\Vert u\right\Vert _{\beta (x)}=1,\text{ }u\in Z_{k}\right\} ,
\end{equation*}%
then $\underset{k\longrightarrow \infty }{\lim }\alpha _{k}=0.$
\end{lemma}

\begin{proof}
Using the continuous embedding $W_{a,b}^{1,p(.)}\left( \Omega \right)
\hookrightarrow L_{b}^{q(.)}(\Omega )$ and the method in (Lemma 4.9, \cite%
{FanZ}), then we get $\lim_{k\longrightarrow \infty }\alpha _{k}=0.$
\end{proof}

Now, we are ready to give our main result.

\begin{theorem}
Let $p^{+}<q^{-}.$ Then $J_{\lambda }$ has a sequence of critical point $%
\left\{ u_{n}\right\} $ such that $J_{\lambda }\left( u_{n}\right)
\longrightarrow +\infty $ and the problem (\ref{P}) has infinite many pairs
of solutions.
\end{theorem}

\begin{proof}
We will use the Fountain Theorem (see \cite[Theorem 3.6]{Wil}). It is well
known that $J_{\lambda }$ is an even functional and fulfills the (PS)
condition. Now, we will prove that if $k$ is large enough, then there are $%
\eta _{k}>\gamma _{k}>0$ such that%
\begin{equation*}
\left( A_{1}\right) \text{ \ \ }b_{k}:=\inf \left\{ J_{\lambda }\left(
u\right) \left\vert u\in Z_{k},\text{ }\left\Vert u\right\Vert _{\beta
(x)}=\gamma _{k}\right. \right\} \longrightarrow \infty \text{ as }%
k\longrightarrow \infty
\end{equation*}%
\begin{equation*}
\left( A_{2}\right) \text{ \ \ }a_{k}:=\max \left\{ J_{\lambda }\left(
u\right) \left\vert u\in Y_{k},\text{ }\left\Vert u\right\Vert _{\beta
(x)}=\eta _{k}\right. \right\} \leq 0.
\end{equation*}

$\left( A_{1}\right) $ For any $u\in Z_{k}$ such that $\left\Vert
u\right\Vert _{\beta (x)}=\gamma _{k}>1$, we have%
\begin{eqnarray*}
J_{\lambda }\left( u\right) &\geq &\frac{1}{p^{+}}I_{\beta (x)}(u)-\frac{%
\lambda }{q^{-}}\dint\limits_{\Omega }b(x)\left\vert u\right\vert ^{q\left(
x\right) }dx \\
&\geq &\frac{1}{p^{+}}I_{\beta (x)}(u)-\frac{\lambda }{q^{-}}\max \left\{
\left\Vert u\right\Vert _{q(.),b}^{q^{-}},\left\Vert u\right\Vert
_{q(.),b}^{q^{+}}\right\} \\
&\geq &\frac{1}{p^{+}}\left\Vert u\right\Vert _{\beta (x)}^{p^{-}}-\frac{%
\lambda }{q^{-}}\max \left\{ \left\Vert u\right\Vert _{q\left( .\right)
,b}^{q^{-}},\left\Vert u\right\Vert _{q\left( .\right) ,b}^{q^{+}}\right\} \\
&\geq &\left\{ 
\begin{array}{c}
\frac{1}{p^{+}}\left\Vert u\right\Vert _{\beta (x)}^{p^{-}}-\lambda ,\text{
\ \ }\left\Vert u\right\Vert _{q\left( .\right) ,b}\leq 1 \\ 
\frac{1}{p^{+}}\left\Vert u\right\Vert _{\beta (x)}^{p^{-}}-\lambda \alpha
_{k}^{q^{+}}\left\Vert u\right\Vert _{\beta (x)}^{q^{+}},\text{ \ \ }%
\left\Vert u\right\Vert _{q\left( .\right) ,b}>1%
\end{array}%
\right. \\
&\geq &\frac{1}{p^{+}}\left\Vert u\right\Vert _{\beta (x)}^{p^{-}}-\lambda
\alpha _{k}^{q^{+}}\left\Vert u\right\Vert _{\beta (x)}^{q^{+}} \\
&=&\frac{1}{p^{+}}\gamma _{k}^{p^{-}}-\lambda \alpha _{k}^{q^{+}}\gamma
_{k}^{q^{+}}
\end{eqnarray*}%
If we choose $\gamma _{k}=\left( \lambda q^{+}\alpha _{k}^{q^{+}}\right) ^{%
\frac{1}{p^{-}-q^{+}}}$, then we get 
\begin{eqnarray*}
J_{\lambda }\left( u\right) &\geq &\frac{1}{p^{+}}\left( \lambda q^{+}\alpha
_{k}^{q^{+}}\right) ^{\frac{p^{-}}{p^{-}-q^{+}}}-\lambda \alpha
_{k}^{q^{+}}\left( \lambda q^{+}\alpha _{k}^{q^{+}}\right) ^{\frac{q^{+}}{%
p^{-}-q^{+}}} \\
&=&\left( \frac{1}{p^{+}}-\frac{1}{q^{+}}\right) \left( \lambda q^{+}\alpha
_{k}^{q^{+}}\right) ^{\frac{p^{-}}{p^{-}-q^{+}}}\longrightarrow \infty
\end{eqnarray*}%
as $k\longrightarrow \infty $ because $p^{+}<q^{+}$ and $\alpha
_{k}\longrightarrow 0.$

$\left( A_{2}\right) $ Let $u\in Y_{k}$ be such that $\left\Vert
u\right\Vert _{\beta (x)}=\eta _{k}>\gamma _{k}>1$. Then, we have 
\begin{eqnarray*}
J_{\lambda }\left( u\right) &\leq &\frac{1}{p^{-}}\left\Vert u\right\Vert
_{\beta (x)}^{p^{+}}-\frac{\lambda }{q^{+}}\dint\limits_{\Omega }\left\vert
u\right\vert ^{q\left( x\right) }b\left( x\right) dx \\
&\leq &\frac{1}{p^{-}}\left\Vert u\right\Vert _{\beta (x)}^{p^{+}}-\frac{%
\lambda }{q^{+}}\min \left\{ \left\Vert u\right\Vert _{q\left( .\right)
,b}^{q^{-}},\left\Vert u\right\Vert _{q\left( .\right) ,b}^{q^{+}}\right\} .
\end{eqnarray*}%
Since the space $Y_{k}$ has finite dimension, then the norms $\left\Vert
.\right\Vert _{\beta (x)}$ and $\left\Vert .\right\Vert _{q\left( .\right)
,b}$ are equivalent. Finally, we obtain 
\begin{equation*}
J_{\lambda }\left( u\right) \rightarrow -\infty \text{ as }\left\Vert
u\right\Vert _{\beta (x)}\rightarrow +\infty \text{, }u\in Y_{k}
\end{equation*}%
due to $p^{+}<q^{-}$.
\end{proof}

\end{document}